\documentclass[11pt]{article}
\usepackage[a4paper, total={6in, 8in}]{geometry}
\usepackage{enumerate}
\usepackage{amsmath}
\usepackage{amssymb,latexsym}
\usepackage{amsthm}
\usepackage{color}
\usepackage{cancel}
\usepackage{graphicx}
\usepackage{cite}
\usepackage{changes}
\definechangesauthor[name=Daniele, color=red]{DAN}
\definechangesauthor[name=Giovanni, color=blue]{GIO}
\usepackage{lineno}
\usepackage{hyperref}
\usepackage{comment}
\usepackage{amscd}

\newtheorem{theorem}{Theorem}[section]

\newtheorem{lemma}[theorem]{Lemma}
\newtheorem{corollary}[theorem]{Corollary}

\newtheorem{proposition}[theorem]{Proposition}

\newtheorem{remark}[theorem]{Remark}

\newcommand{\fqn}{\mathbb{F}_{q^n}}

\newcommand{\cS}{{\mathcal S}}

\newcommand{\F}{{\mathbb F}}

\newcommand{\fq}{{\mathbb F}_{q}}

\newcommand{\la}{\langle}
\newcommand{\ra}{\rangle}
\renewcommand{\mod}{\hbox{{\rm mod}\,}}

\newcommand{\PG}{\mathrm{PG}}
\newcommand{\N}{\mathrm{N}}

\title{Investigating the exceptionality of scattered polynomials}
\author{Daniele Bartoli\thanks{Dipartimento di Matematica e Informatica, Universit\`a degli studi di Perugia,  Perugia, Italy. daniele.bartoli@unipg.it}, 
Giovanni Zini\thanks{Dipartimento di Matematica e Fisica, Universit\`a degli Studi della Campania ``Luigi Vanvitelli'', Caserta, Italy.
giovanni.zini@unicampania.it}, and 
Ferdinando Zullo\thanks{Dipartimento di Matematica e Fisica, Universit\`a degli Studi della Campania ``Luigi Vanvitelli'', Caserta, Italy.
ferdinando.zullo@unicampania.it}
}
\date{ }

\begin{document}

\maketitle
\begin{abstract}
    Scattered polynomials over a finite field $\mathbb{F}_{q^n}$ have been introduced by Sheekey in 2016, and a central open problem regards the classification of those that are \emph{exceptional}. So far, only two families of exceptional scattered polynomials are known.
Very recently, Longobardi and Zanella weakened the property of being scattered by introducing the notion of \emph{L-$q^t$-partially scattered} and \emph{R-$q^t$-partially scattered} polynomials, for $t$ a divisor of $n$. Indeed, a polynomial is scattered if and only if it is both L-$q^t$-partially scattered and R-$q^t$-partially scattered.
    In this paper, by using techniques from algebraic geometry over finite fields and function fields theory, we show that the property which is is the hardest to be preserved  is the L-$q^t$-partially scattered one. On the one hand, we are able to extend the classification results  of exceptional scattered polynomials to exceptional L-$q^t$-partially scattered polynomials. On the other hand, the R-$q^t$-partially scattered property seems  more stable. We present a  large family of R-$q^t$-partially scattered polynomials, containing examples of exceptional R-$q^t$-partially scattered polynomials, which turn out to be connected with linear sets of so-called \emph{pseudoregulus type}.
    In order to detect new examples of polynomials which are R-$q^t$-partially scattered, we introduce two different notions of equivalence preserving this property and concerning natural actions of the groups ${\rm \Gamma L}(2,q^n)$ and ${\rm \Gamma L}(2n/t,q^t)$.
    In particular, our family contains many examples of inequivalent polynomials, and geometric arguments are used to determine the equivalence classes under the action of ${\rm \Gamma L}(2n/t,q^t)$.
\end{abstract}



\section{Introduction}

Let $q$ be a prime power, $n$ be a positive integer, and $f(x)=\sum_{i=0}^{k} a_i x^{q^i}\in\fqn[x]$ be an $\fq$-linearized polynomial over the finite field $\fqn$.
We also   assume that the $q$-degree $k$ of $f(x)$ is smaller than $n$, so that the identification with the map $x\mapsto f(x)$ defines a one-to-one correspondence between such polynomials and  $\fq$-linear maps over $\fqn$.

An $\fq$-linearized polynomial $f(x)\in\fqn[x]$ is said to be \emph{scattered} of index $\ell\in\{0,\ldots,n-1\}$ over $\fqn$ if, for any $y,z\in\fqn^*$,
\begin{equation}\label{eq:scatt}
\frac{f(y)}{y^{q^\ell}}=\frac{f(z)}{z^{q^\ell}} \Longrightarrow \frac{y}{z}\in \fq;
\end{equation}
see \cite{BZ}.
Scattered polynomials $f(x)\in\fqn[x]$ yield \emph{scattered subspaces} $U_f$ (w.r.t. a Desarguesian spread) in $\fqn\times\fqn$ by defining
\[
U_f=\{(x^{q^\ell},f(x))\colon x\in\fqn\}.
\]
Scattered subspaces of maximum dimension have many applications, such as translation hyperovals \cite{Glynn}, translation caps in affine spaces \cite{BGMP2015}, two-intersection sets \cite{BL2000}, blocking sets \cite{BBL2000}, translation spreads of the Cayley
generalized hexagon \cite{MP2015}, finite semifields \cite{LavrauwPolverino}, coding theory \cite{PZScatt,ZiniZulloScatt}, and graph theory \cite{CK}.

Starting from \cite{BZ}, a much stronger property regarding scattered polynomials, namely their exceptionality, has been defined and deeply investigated.
An $\fq$-linearized polynomial $f(x)\in\fqn[x]$ is said to be \emph{exceptional scattered} of index $\ell\in\{0,\ldots,n-1\}$ if there exist infinitely many $m\in\mathbb{N}$ such that, for any $y,z\in\mathbb{F}_{q^{nm}}$, Condition \eqref{eq:scatt} holds.

While several families of scattered polynomials have been constructed in recent years, only two families of exceptional scattered polynomials are known:
\begin{itemize}
    \item $f(x)=x^{q^s}$ of index $0$, with $\gcd(s,n)=1$ (polynomials of so-called pseudoregulus type);
    \item $f(x)=x+\delta x^{q^{2s}}$ of index $s$, with $\gcd(s,n)=1$ and $\mathrm{N}_{q^n/q}(\delta)\ne1$ (so-called LP polynomials).
\end{itemize}
Several tools have already been proposed in the study of exceptional scattered polynomials, related to certain algebraic curves or Galois extensions of function fields; see \cite{SurveryDani,BZ,BM,FM}.
However, their classification is still unknown when the index is greater than $1$.
In this paper we investigate the exceptional scatteredness of a polynomial by considering separately the exceptionality of two weaker properties defined in \cite{LZ}, namely the \emph{L-$q^t$-partial scatteredness} and the \emph{R-$q^t$-partial scatteredness}.

Let $f(x)$ be an $\fq$-linearized polynomial over $\fqn$, $t$ be a divisor of $n$, and $\ell\in\{0,\ldots,n-1\}$.
We say that $f(x)$ is \emph{L-$q^t$-partially scattered of index $\ell$} if for any $y,z\in\fqn^*$,
\begin{equation}\label{eq:condL}
\frac{f(y)}{y^{q^\ell}}=\frac{f(z)}{z^{q^\ell}}\Longrightarrow \frac{y}{z}\in\mathbb{F}_{q^t},
\end{equation}
and that $f(x)$ is \emph{R-$q^t$-partially scattered of index $\ell$} if for any  $y,z\in\fqn^*$,
\begin{equation}\label{eq:condR}
\frac{f(y)}{y^{q^\ell}}=\frac{f(z)}{z^{q^\ell}}\,\mbox{ and }\, \frac{y}{z}\in\mathbb{F}_{q^t}\Longrightarrow \frac{y}{z}\in\fq.
\end{equation}
From now on, whenever the index is not specified, we mean $\ell=0$.

We say that $f(x)$ is \emph{exceptional L-$q^t$-partially scattered of index $\ell$} (resp. \emph{exceptional R-$q^t$-partially scattered of index $\ell$}) if there exist infinitely many $m\in\mathbb{N}$ such that Condition \eqref{eq:condL} (resp. Condition \eqref{eq:condR}) holds for any $y,z\in\mathbb{F}_{q^{nm}}^*$.
Clearly, for any $t$, $f(x)$ is (exceptional) scattered of index $\ell$ if and only if it is (exceptional) both L- and R-$q^t$-partially scattered of index $\ell$.

Some results on, and characterizations of L- and R-$q^t$-partially scattered polynomials have been provided in \cite{LZ}; see Section \ref{sec:preliminaries}.

In this paper, we start the investigation of such properties by those families that contain the known examples of exceptional scattered polynomials, namely the monomials $x^{q^u}$ of index $0$ and the LP polynomials $x+\delta x^{q^{2s}}$ of index $s$.
In this way we obtain exceptional L- and R-$q^t$-partially scattered monomials which are not scattered, and other results for L- and R-$q^t$-partially scattered polynomials of LP type; see Section \ref{sec:first}.

Afterwards, we prove in Section \ref{sec:exceptionality} several necessary conditions for a polynomial to be exceptional L-$q^t$-partially scattered, classifying for instance those of index at most $1$.
This is done by means of tools  already  exploited in the literature in connection with the exceptional scatteredness, such as algebraic curves and function fields over finite fields, also exploiting a method due to G. Micheli in \cite{Micheli1,Micheli2}. Interestingly, such connections can be generalized to exceptional L-$q^t$-partial scatteredness.

Turning to the R- side of $q^t$-partial scatteredness, we detect in Section \ref{sec:example} an explicit large family of R-$q^t$-partially scattered polynomials of the form
\begin{equation}\label{eq:familyintro}
f(x)=\sum_{i=0}^{n/t-1} a_i x^{q^{it+s}} \in \mathbb{F}_{q^{n}}[x],
\end{equation}
which extends previously known examples and whose exact number we are able to count.
As a byproduct, we provide  families of exceptional R-$q^t$-partially scattered binomials of the shape $x^{q^{kt+s}}+\alpha x^{q^s}$.
Such a construction suggests that the existence of exceptional L-$q^t$-partially scattered polynomials is much harder to prove than that of exceptional R-$q^t$-partially scattered ones.

We further investigate the family \eqref{eq:familyintro}  in Section \ref{sec:pseudo} from a geometric point of view: indeed, this family can be alternatively constructed by considering linear sets of pseudoregulus type in the projective space ${\rm PG}(2n/t-1,q^t)$.

The search for new R-$q^t$-partially scattered polynomials naturally requires the investigation of the equivalence issue, in the sense of a suitable group action on the polynomials $f(x)\in\mathbb{F}_{q^{tt^\prime}}[x]$ preserving the desired property.
To this aim, we analyse in Section \ref{sec:equiv} the equivalence defined by a natural action of the group ${\rm \Gamma L}(2,q^{tt^\prime})$ on the elements $(x^{q^\ell},f(x))$. In Section \ref{sec:weaker} we study a weaker equivalence defined by a natural action of the larger group ${\rm \Gamma L}(2t^\prime,q^t)$.
We solve the weak equivalence issue for the family of R-$q^t$-partially scattered polynomials described in Section \ref{sec:example}.
Some open problems in different directions conclude the paper.

\section{Preliminaries}\label{sec:preliminaries}

Let $V$ be an $r$-dimensional $\F_{q^n}$-vector space and let $\cS$ be an $n$-spread of $V$, seen as an $\F_q$-vector space.
An $\F_q$-subspace $U$ of $V$ is called \emph{scattered} w.r.t. $\mathcal{S}$ if $U$ meets every element of $\mathcal{S}$ in an $\F_q$-subspace of dimension at most one; see \cite{BL2000}.
If we consider $V$ as an $rn$-dimensional $\F_q$-vector space, then 
\[ \{\langle \mathbf{v}\rangle_{\fqn} \colon \mathbf{v}\in V\setminus\{\mathbf{0}\}\} \]
is an $n$-spread of $V$, called a  \emph{Desarguesian spread}. 
In this paper we always study scattered $\fq$-subspaces  w.r.t. this  Desarguesian spread,  simply called scattered subspaces. 
For such subspaces Blokhuis and Lavrauw showed that their dimension is bounded above by $rn/2$ and it is now known that when $rn$ is even   there always exist scattered subspaces of this dimension \cite{BBL2000, BGMP2015, BL2000, CSMPZ2016}. 

Recently, much focus has been placed on  scattered $\fq$-subspaces of dimension $n$ in $V=\fqn\times \fqn$, especially because of their connections with MRD codes; see \cite{John}. For any $\fq$-subspace $U$ of dimension $n$ in  $\fqn\times \fqn$ and any non-negative integer $\ell<n$ there exist a basis 
and an $\fq$-linearized polynomial over $\fqn$, that is a polynomial $f(x)$ of the form $\sum_{i=0}^{n-1}a_ix^{q^i}\in\fqn[x]$, such that
\[ U=U_f=\{(x^{q^\ell},f(x)) \colon x \in \mathbb{F}_{q^n}\}. \]
When $U_f$ is scattered,  we say that $f(x)$ is a \emph{scattered polynomial} of index $\ell$.
Scattered polynomials were introduced in \cite{John} for $\ell=0$, and in \cite{BZ} for any $\ell$; note that this definition is equivalent to the one in Equation \eqref{eq:scatt}.

In \cite{LZ}, the authors weaken the property of being scattered for a polynomial as in Equations \eqref{eq:condL} and \eqref{eq:condR}.
In the following we resume the results contained in \cite{LZ} which will be useful for our purposes.
For any $\fq$-linearized polynomial $f(x)\in \fqn[x]$ and any $\rho \in \fqn^*$, define 
\[f_{\rho}(x):=f(\rho x)-\rho f(x).\]

\begin{proposition}\cite[Proposition 2.6]{LZ}\label{prop:characterizationpartially}
\begin{enumerate}
    \item An $\fq$-linearized polynomial $f(x)$ is R-$q^t$-partially scattered if and only if the map $f_{\rho}(x)$ is bijective over $\fqn$ for any $\rho \in \F_{q^t}\setminus\fq$.
    \item An $\fq$-linearized polynomial $f(x)$ is L-$q^t$-partially scattered if and only if the map $f_{\rho}(x)$ is bijective over $\fqn$ for any $\rho \in \F_{q^n}\setminus\F_{q^t}$.
    \item An $\fq$-linearized polynomial $f(x)$ is scattered if and only if the map $f_{\rho}(x)$ is bijective over $\fqn$ for any $\rho \in \F_{q^n}\setminus\fq$.
\end{enumerate}
\end{proposition}

It is possible to translate the property of being R-$q^t$-partially scattered in terms of scattered subspaces.

\begin{theorem}\cite[Theorem 2.3]{LZ}\label{th:scatteredbig}
Let $f(x)\in\fqn[x]$ be an $\mathbb{F}_q$-linearized polynomial and define $U_f=\{(y,f(y)) \colon y \in \F_{q^n}\}$.
The polynomial $f(x)$ is R-$q^t$-partially scattered  
if and only if $U_f$ is a scattered $\mathbb{F}_q$-subspace of the $\mathbb{F}_{q^t}$-vector space $\mathbb{F}_{q^n}\times \mathbb{F}_{q^n}$.
\end{theorem}

For any $\mathbf{a}=(a_0,\ldots,a_{t'-1})\in \fqn^{t'}$ define $g_{\mathbf{a}}(x)=\sum_{i=0}^{t'-1} a_i x^{q^{it}}\in\fqn[x]$. It is possible to construct new R-$q^t$-partially scattered polynomials from an already known one.

\begin{proposition}\cite[Proposition 2.10]{LZ}\label{prop:LZ}
For any R-$q^t$-partially scattered $\fq$-linearized polynomial $\varphi(x)$, the polynomial $f(x)=(g_{\mathbf{a}}\circ \varphi)(x)$ is R-$q^t$-partially scattered if and only if $g_{\mathbf{a}}(x)$ is invertible over $\fqn$.
\end{proposition}

Consider the non-degenerate symmetric bilinear form of $\F_{q^n}$ over $\F_q$ defined  by
\begin{equation*}
\la x,y\ra= \mathrm{Tr}_{q^n/q}(xy),
\end{equation*}
for every $x,y \in \F_{q^n}$.

\noindent The \emph{adjoint} $\hat{f}$ of the $\fq$-linearized polynomial $\displaystyle f(x)=\sum_{i=0}^{n-1} a_ix^{q^i} \in \F_{q^n}[x]$ with respect to the bilinear form $\la\cdot,\cdot\ra$, i.e. the unique function over $\fqn$ satisfying
\[ \mathrm{Tr}_{q^n/q}(yf(z))=\mathrm{Tr}_{q^n/q}(z\hat{f}(y)) \]
for every $y,z \in \F_{q^n}$, is given by
\begin{equation}\label{eq:adj} \hat{f}(x)=\sum_{i=0}^{n-1} a_i^{q^{n-i}}x^{q^{n-i}}.
\end{equation}

The adjoint operation preserves both the properties of being L-$q^t$-partially and R-$q^t$-partially scattered.

\begin{proposition}\cite[Proposition 2.20]{LZ}\label{prop:adj}
Let $f(x)$ be a $\fq$-linearized polynomial over $\fqn$. Then $f(x)$ is L-$q^t$-partially scattered (resp. R-$q^t$-partially scattered) if and only if $\hat{f}(x)$ is L-$q^t$-partially scattered (resp. R-$q^t$-partially scattered).
\end{proposition}

\section{First examples}\label{sec:first}

In this section we characterize monomials which are L-$q^t$-partially or R-$q^t$-partially scattered and we deal with the well known class of LP-polynomials.

\begin{proposition}\label{prop:mon}
Let $u\geq1$ and $f(x)=x^{q^u} \in \fqn[x]$.
Then
\begin{itemize}
    \item $f(x)$ is L-$q^t$-partially scattered if and only if $\gcd(u,n) \mid t$;
    \item $f(x)$ is R-$q^t$-partially scattered if and only if $\gcd(u,t)=1$;
    \item $f(x)$ is scattered if and only if $\gcd(u,n)=1$.
\end{itemize}
\end{proposition}
\begin{proof}
For any $\rho\in\mathbb{F}_{q^n}$, $f_\rho(x)=f(\rho x)-\rho f(x)=(\rho^{q^u}-\rho)x^{q^u}$ is bijective if and only if $\rho^{q^u}-\rho\ne 0$, i.e. $\rho\notin\mathbb{F}_{q^u}\cap\mathbb{F}_{q^n}=\mathbb{F}_{q^{\gcd(u,n)}}$.
So, by Proposition \ref{prop:characterizationpartially}, $f(x)$ is L-$q^t$-partially scattered if and only if $\mathbb{F}_{q^{\gcd(u,n)}}\subseteq\mathbb{F}_{q^t}$, i.e. $\gcd(u,n)\mid t$. Also, $f(x)$ is R-$q^t$-partially scattered if and only if $\mathbb{F}_{q^{\gcd(u,n)}}\cap\mathbb{F}_{q^t}=\mathbb{F}_{q}$, i.e. $\gcd(u,t)=1$. Thus, 
$f(x)$ is scattered if and only if $\gcd(u,n)=1$.
\end{proof}

Therefore, a characterization of exceptional partially scattered monomials is obtained.

\begin{corollary}
Let $u\geq1$ and $f(x)=x^{q^u} \in \fqn[x]$.
\begin{itemize}
    \item $f(x)$ is exceptional L-$q^t$-partially scattered, but not scattered, if and only if $1\ne \gcd(u,n) \mid t$;
    \item $f(x)$ is exceptional R-$q^t$-partially scattered, but not scattered, if and only if $1=\gcd(u,t)<\gcd(u,n)$.
\end{itemize}
\end{corollary}

Let $f(x)=x^{q^{s(n-1)}}+\delta x^{q^s}\in\fqn[x]$, where $s$ satisfies $\gcd(s,n)=1$. The polynomial $f(x)$ is called \emph{LP-polynomial} after Lunardon and Polverino who first introduced it in \cite{LunPol2000}.

After a series of papers, Zanella in \cite{Zanella} characterized those $\delta \in \fqn$ for which $f(x)$ is scattered and hence both L-$q^t$-partially and R-$q^t$-partially scattered.

\begin{theorem}{\rm \cite[Theorem 3.4]{Zanella}}
Let $\delta\in\mathbb{F}_{q^n}$, $\gcd(s,n)=1$, $f(x)=x^{q^{s(n-1)}}+\delta x^{q^s}\in\fqn[x]$.
The polynomial $f(x)$ is scattered if and only if $\N_{q^n/q}(\delta)\ne 1$.
\end{theorem}



The proof of the next proposition follows the ones of \cite[Lemma 4.4]{LMPT2015} and \cite[Proposition 7.4]{BMZZ}.

\begin{proposition}\label{prop:LP}
Let $n=tt'$ with $t,t' \in \mathbb{N}$, $\delta\in\mathbb{F}_{q^n}$, $\gcd(s,n)=1$, $f(x)=x^{q^{s(n-1)}}+\delta x^{q^s}\in\fqn[x]$.
Assume that $n$ is odd and $\N_{q^n/q}(\delta)=1$.
Then $f(x)$ is neither L-$q^t$-partially scattered nor R-$q^t$-partially scattered.
\end{proposition}
\begin{proof}
Let $\rho \in \F_{q^t}\setminus \fq$. 
The identity $\frac{f(x)}x=\frac{f(\rho x)}{\rho x}$ reads
\[ \rho (x^{q^{s(n-1)}}+\delta x^{q^s})= (\rho x)^{q^{s(n-1)}}+\delta (\rho x)^{q^s}, \]
that is 
\begin{equation}\label{eq:LPcond} x^{q^{2s}-1}=\frac{1}{\delta^{q^s}(\rho-\rho^{q^s})^{q^s-1}}. \end{equation}
Since $n$ is odd and $\N_{q^n/q}(\delta)=1$, there exists $x_0\in \fqn^*$ satisfying Equation \eqref{eq:LPcond}. So, $f(x)$ is not R-$q^t$-partially scattered since $\rho \notin \fq$.
Similar arguments show that $f(x)$ is not L-$q^t$-partially scattered.
\end{proof}

\section{Exceptional L-\texorpdfstring{$q^t$}{Lg}-partially scattered polynomials}\label{sec:exceptionality}

Inspired by the results in \cite{BM,BZ}, we provide an inequality involving the parameters of an L-$q^t$-partially scattered polynomial. As a byproduct, we obtain a non-existence result for exceptional L-$q^t$-partially scattered polynomials.

\begin{remark}
In the study of L-$q^t$-partially scattered polynomials of index $\ell$, we can assume that $f(x)$ is $\ell$-normalized, in the following sense (see \cite[p. 511]{BZ}):
\begin{itemize}
    \item[(i)] the  $q$-degree $k$ of $f(x)$ is smaller than $n$;
    \item[(ii)] $f(x)$ is monic;
    \item[(iii)] the coefficient $a_t$ of $x^{q^t}$ in $f(x)$ is zero;
    \item[(iv)] if $\ell>0$, then the coefficient of $x$ in $f(x)$ is nonzero, i.e. $f(x)$ is separable.
\end{itemize}
In fact, $f(x)$ is L-$q^t$-partially scattered of index $\ell$ if and only if the monic polynomial $\frac{1}{a_k}f(x)$ is L-$q^t$-partially scattered of index $\ell$.
Also, $\frac{f(x)}{x^{q^\ell}}$ equals $\frac{f(x)-a_\ell x^{q^\ell}}{x^{q^\ell}}+a_\ell$, so that $f(x)$ is L-$q^t$-partially scattered of index $\ell$ if and only if $f(x)-a_\ell x^{q^\ell}$ is L-$q^t$-partially scattered of index $\ell$.
Finally, if $\ell >0$ and $q^v$ is the smallest degree of a monomial in $f(x)$, then $\frac{f(x)}{x^{q^\ell}}=\left(\frac{\sum a_i^{q^{n-v}}x^{q^{i-v}}}{x^{q^{\ell-v}}}\right)^{q^v}$. Therefore, $f(x)$ is L-$q^t$-partially scattered of index $\ell$ if and only if $\sum {a_i}^{q^{n-v}}x^{q^{i-v}}$ is L-$q^t$-partially scattered of index $\ell-v$.
\end{remark}

The following useful lemma generalizes \cite[Lemma 2.1]{BZ} and is obtained in a similar way. 

\begin{lemma}\label{lemma:importante}
Let $\mathcal{C}$ be the plane curve with affine equation
\begin{equation}\label{eq:curve}
\mathcal{C}\colon\quad \frac{f(X)Y^{q^\ell}-f(Y)X^{q^\ell}}{X^q Y - X Y^q}=0.
\end{equation}
The polynomial $f(x)$ is L-$q^t$-partially scattered of index $\ell$ if and only if every affine $\fqn$-rational point $(\bar{x},\bar{y})\in\mathcal{C}$ with $\bar{x},\bar{y}\ne0$ satisfies $\bar{y}/\bar{x}\in\mathbb{F}_{q^t}$.
\end{lemma}

Now we can prove a bound on the parameters of L-$q^t$-partially scattered polynomials.

\begin{theorem}\label{th:inequality}
Let $f(x)=\sum_{j=0}^{k}a_j x^{q^j}\in\fqn[x]$ be a non-monomial $\ell$-normalized L-$q^t$-partially scattered polynomial, and let $v=\min\{j\colon a_j\ne0\}$.

If $\ell=0$, then
\[
q^n - (q^k-q-1)(q^k-q-2)\sqrt{q^n} - q^t(q^k-q)- 2(q^{k-v}-1)\leq 0 .
\]
If one of the following two conditions holds:
\begin{itemize}
    \item[(i)] $\ell=1$ and $k\geq 3$;
    \item[(ii)] $\ell\geq2$, with $k \geq 3\ell$ if  $\ell \mid k$ or $k \geq 2\ell-1$ if $\ell \nmid k$, and  either
    \begin{itemize}
        \item $f(x)=a_0 x+a_1 x^q+\sum_{j>\ell}a_j x^{q^j}$ with $k \geq \ell+2$, or 
        \item $f(x)=a_0 x+\sum_{j>\ell}a_j x^{q^j}$;
\end{itemize}
\end{itemize}
then
\[
q^n - (q^k+q^\ell-q-2)(q^k+q^\ell-q-3)\sqrt{q^n} - q^t(q^k+q^\ell-q-1) \leq 0.
\]
\end{theorem}

\begin{proof}
As proved in \cite{BM,BZ} (see in particular \cite[Proposition 3.6 and Section 4]{BM} and \cite[Theorems 3.1 and 3.5]{BZ}), the curve $\mathcal{C}$ defined as in \eqref{eq:curve} contains an absolutely irreducible component $\mathcal{X}$ defined over $\fqn$.
By the Hasse-Weil bound,
the number $N_{q^n}$ of $\fqn$-rational points of $\mathcal{X}$ satisfies
\[
N_{q^n}\geq q^n+1 - (q^k+q^\ell-q-2)(q^k+q^\ell-q-3)\sqrt{q^n}.
\]
Since $f(x)$ is L-$q^t$-partially scattered of index $\ell$, Lemma \ref{lemma:importante} implies that the $\fqn$-rational points of $\mathcal{C}$ lie either on the line at infinity $r_\infty$, or on the line $r_X:X=0$, or on the line $r_Y:Y=0$, or on the line $r_{\lambda}:Y=\lambda X$ for some $\lambda\in\mathbb{F}_{q^t}^*$.

As shown in the proof of \cite[Theorem 3.3]{BMZZ}, the origin $(0,0)$ is a point of $\mathcal{C}$ if and only if either $\ell=0$ and $v>1$, or $\ell>1$.
Also, the number of affine $\fqn$-rational points of $((\mathcal{C}\cap r_X)\cup(\mathcal{C}\cap r_Y))\setminus\{(0,0)\}$ is zero if $\ell>0$ and at most $2(q^{k-v}-1)$ if $\ell=0$.
Finally, for every $\lambda\in\mathbb{F}_{q^t}^*$, we have $|\mathcal{C}\cap r_{\lambda}|\leq q^k+q^\ell-q-1$, as well as $|\mathcal{C}\cap r_{\infty}|\leq q^k+q^\ell-q-1$.

Therefore,
if $\ell=0$ then
\[
q^n+1 - (q^k-q-1)(q^k-q-2)\sqrt{q^n} \leq q^t(q^k-q)+ 2(q^{k-v}-1)+\chi_{v>1},
\]
where $\chi_{v>1}$ is $1$ if $v>1$ and $0$ otherwise.
If $\ell>0$ then
\[
q^n+1 - (q^k+q^\ell-q-2)(q^k+q^\ell-q-3)\sqrt{q^n} \leq q^t(q^k+q^\ell-q-1)+\chi_{v>1}.
\]
The claim follows.
\end{proof}

By direct computation, Theorem \ref{th:inequality} has the following consequence.

\begin{corollary}\label{cor:nuncestanno}
Let $f(x)=\sum_{j=0}^{k}a_j x^{q^j}\in\fqn[x]$ be a non-monomial $\ell$-normalized L-$q^t$-partially scattered polynomial.
Suppose that one of the following conditions holds:
\begin{itemize}
    \item $\ell=0$;
    \item $\ell=1$ and $k\geq 3$;
    \item $\ell\geq2$, with $k \geq 3\ell$ if  $\ell \mid k$ or $k \geq 2\ell-1$ if $\ell \nmid k$, and  either
    \begin{itemize}
        \item $f(x)=a_0 x+a_1 x^q+\sum_{j>\ell}a_j x^{q^j}$ with $k \geq \ell+2$, or 
        \item $f(x)=a_0 x+\sum_{j>\ell}a_j x^{q^j}$.
    \end{itemize}
\end{itemize}
Then
\[
\frac{n}{2}\leq \max\left\{2k,2\ell,\frac{k+t}{2},\frac{\ell+t}{2}\right\}.
\]
In particular, $f(x)$ is not exceptional L-$q^t$-partially scattered.
\end{corollary}

As a consequence of Corollary \ref{cor:nuncestanno} and Proposition \ref{prop:mon}, we get a classification of exceptional L-$q^t$-partially scattered polynomials of small index.

\begin{corollary}
Let $f(x)\in\fqn[x]$ be an $\ell$-normalized exceptional L-$q^t$-partially scattered polynomial. Then the following holds:
\begin{itemize}
    \item if $\ell=0$, then $f(x)=x^{q^s}$ with $\gcd(s,n)\mid t$;
    \item if $\ell=1$, then $f(x)=\delta x + x^{q^2}$ with $\delta \in \fqn$. 
\end{itemize}
\end{corollary}

When $n$ is odd, Proposition \ref{prop:LP} implies the following result.

\begin{corollary}
Let $n$ be an odd positive integer and $f(x)\in\fqn[x]$ be a $1$-normalized exceptional L-$q^t$-partially scattered polynomial. Then $f(x)=\delta x + x^{q^2}$ with $\mathrm{N}_{q^n/q}(\delta)\ne1$.
\end{corollary}

Exceptional L-$q^t$-partially scattered polynomials $f(x)$ can also be investigated in connection with the Galois groups of certain extensions of function fields, similarly to what is done in \cite{BMZZ,FM} for exceptional scattered polynomials.

For the notations we refer to \cite[Section 2]{FM}.
In particular, for an $\ell$-normalized $\fq$-linearized polynomial $f(x)\in\fqn[x]$, let $s$ be a transcendental over $\fqn$ and $M$ be the splitting field of $f(x)-sx^{q^{\ell}}$ over $\fqn(s)$.
For any positive integer $m$, $M_m$ is the compositum function field $M\cdot \mathbb{F}_{q^{nm}}$, $k_m$ is the field of constants of $M_m$, $G_m^{\rm arith}$ and $G_m^{\rm geom}$ are respectively the arithmetic and the geometric Galois group of $M_m/\mathbb{F}_{q^{nm}}(s)$, and $\varphi_m$ is the isomorphism $G_m^{\rm arith}/G_m^{\rm geom}\to{\rm Gal}(k_m/\mathbb{F}_{q^{nm}})$.
Moreover, there exists a constant $C>0$ depending on $M/\fqn(s)$ such that for any $m$ satisfying $q^{nm}>C$ the following property holds: every $\gamma\in G_m^{\rm arith}$ such that $\varphi_m(\gamma)$ is the Frobenius automorphism for the extension $k_m/\mathbb{F}_{q^{nm}}$ is also a Frobenius at an unramified place at finite of degree $1$ of $\mathbb{F}_{q^{nm}}(s)$.

\begin{theorem}\label{th:corrisp}
Let $\ell\geq1$ and $f(x)\in\fqn[x]$ be an $\ell$-normalized $\fq$-linearized polynomial. Let $d:=\max\{k,\ell\}<n$.
With the above notation, let $m\geq1$ be such that $q^{nm}>C$.
For any positive integer $t<d$, the following are equivalent:
\begin{itemize}
    \item[(i)] for every $z\in\mathbb{F}_{q^{nm}}^*$ there exists a $t$-dimensional $\fq$-subspace $U_z$ of $\mathbb{F}_{q^{nm}}$ such that, for any $y\in\mathbb{F}_{q^{nm}}^*$,
    \begin{equation}\label{eq:extdefL} \frac{f(y)}{y^{q^\ell}}=\frac{f(z)}{z^{q^\ell}} \Longrightarrow \frac{y}{z}\in U_z;
    \end{equation}
    \item[(ii)] for every $\gamma\in G_m^{\rm arith}$ such that $\varphi_m(\gamma)$ is a Frobenius for $k_m/\mathbb{F}_{q^{nm}}$ and every $h\in G_m^{\rm geom}$, the following condition holds:
    \[
    {\rm rk}(h\gamma - {\rm Id})\geq d-t.
    \]
\end{itemize}
In particular,  if $f(x)$ is L-$q^t$-partially scattered, then $G_m^{\rm arith}\neq G_m^{\rm geom}$.
\end{theorem}

\begin{proof}
The proof generalizes the one of \cite[Theorem 2.7]{FM}.

(i) $\Rightarrow$ (ii). Let $h\in G_m^{\rm geom}$ and $\gamma\in G_m^{\rm arith}$ be such that $\varphi_m(\gamma)$ is the Frobenius automorphism for $k_m/\mathbb{F}_{q^{nm}}$. Since $q^{nm}>C$, there exists $s_0\in\mathbb{F}_{q^{nm}}$ such that the degree-$1$  zero $P$ of $s-s_0$ in $\mathbb{F}_{q^{nm}}(s)$ is unramified under $M_m$ and $h\gamma$ is a Frobenius at $P$.
By Condition (i), $f(x)/x-s_0x^{q^{\ell}-1}$ has at most $q^t-1$ roots in $\mathbb{F}_{q^{nm}}$, so that there are at most $q^t-1$ places of degree $1$ in $L_m:=\mathbb{F}_{q^{nm}}[x]/(f(x)/x-sx^{q^{\ell}-1})$ lying over $P$.
Let $R$ be a place of $M_m$ lying over $P$ and $V$ be the $d$-dimensional $\mathbb{F}_q$-vector space of roots of $f(x)-s x^{q^\ell}$.
By \cite[Corollary 2.5]{FM}, the decomposition group $D(R|P)$ has at most $q^t-1$ fixed points on $V\setminus\{0\}$.
As $R$ is unramified over $P$, $D(R|P)$ is cyclic and generated by $h\gamma$.
Then $h\gamma$ has at most $q^t-1$ fixed points on $V\setminus\{0\}$, and hence $h\gamma-{\rm Id}$ has rank at least $d-t$.

(ii) $\Rightarrow$ (i). Suppose by way of contradiction that Condition (i) does not hold. Then there exists $s_0\in\mathbb{F}_{q^{nm}}$ such that $f(x)/x-s_0 x^{q^{\ell}-1}$ has at least $q^{t+1}-1$ roots in $\mathbb{F}_{q^{nm}}$.
Let $P$ be the zero of $s-s_0$ in $\mathbb{F}_{q^{nm}}(s)$.
Then all places at finite of $L_m$ lying over $P$ are unramified (because the polynomial $f(x)-s_0x^{q^\ell}$ is separable), and there are at least $q^{t+1}-1$ of them having degree $1$.
Let $R$ be a place of $M_m$ over $P$.
By \cite[Corollary 2.5]{FM}, $D(R|P)$ has at least $q^{t+1}-1$ fixed points on $V\setminus\{0\}$.
This holds in particular for an element $\gamma\in D(R|P)$ whose image in the surjective homomorphism $D(R|P)\to{\rm Gal}(k_R/\mathbb{F}_{q^{nm}})$ is the Frobenius for $k_m/\mathbb{F}_{q^{nm}}$ (notice that ${\rm Gal}(k_m/\mathbb{F}_{q^{nm}})$ is a subgroup of ${\rm Gal}(k_R/\mathbb{F}_{q^{nm}})$ of index the degree of $R$).
Therefore $\varphi_m(\gamma)$ is the Frobenius of $k_m/\mathbb{F}_{q^{nm}}$, and $\gamma-{\rm Id}$ has rank at most $d-(t+1)$. This is a contradiction to (ii).
\end{proof}

Note that the case $t=1$ is Theorem 2.6 in \cite{FM} and corresponds to $f(x)$ being scattered over $\mathbb{F}_{q^{nm}}$.
Also, L-$q^t$-partially scattered polynomials over $\mathbb{F}_{q^{nm}}$ satisfy Condition \eqref{eq:extdefL}, with $U_z=\mathbb{F}_{q^t}$ for every $z$.

\begin{corollary}
Under the same notation as in Theorem \ref{th:corrisp}, suppose that $d$ is prime. Then the following hold:
\begin{itemize}
    \item if $f(x)$ is not a monomial and $q$ is odd, then $f(x)$ is not exceptional L-$q^t$-partially scattered;
    \item if $f(x)$ is a monomial and $\gcd(d,n)=1$, then $f(x)$ is exceptional scattered;
    \item if $f(x)$ is a monomial and $1<\gcd(d,n)\mid t$, then $f(x)$ is exceptional L-$q^t$-partially scattered but not exceptional scattered.
\end{itemize}
\end{corollary}

\begin{proof}
If $f(x)$ is not a monomial, suppose by way of contradiction that $f(x)$ is exceptional L-$q^t$-partially scattered. Then, by Theorem \ref{th:corrisp}, $G_m^{\rm arith}\ne G_m^{\rm geom}$ for any $m\geq1$ big enough. Then, as in \cite[Lemma 4.1]{FM}, $|G_m^{\rm geom}|=q^d-1$ and $G_m^{\rm arith}\cong GL(1,q^d)\rtimes C_d$, and a contradiction is obtained as in \cite[Page 700]{FM}.
If $f(x)$ is a monomial, then the claim follows from Proposition \ref{prop:mon}.
\end{proof}

It is worth noting that we do not know so far exceptional L-${q^t}$-partially scattered polynomials which are not exceptional scattered, other than in the monomial case.

A possible machinery to obtain suitable examples involves the group-theoretical structure of $G_m^{\rm geom}$ and $G_m^{\rm arith}$, and is currently under investigation in \cite{DanGio2021}.

\section{A family of R-\texorpdfstring{$q^t$}{Lg}-partially scattered polynomials}\label{sec:example}

In this section, via Proposition \ref{prop:LZ}, we consider a new family of R-$q^t$-partially scattered polynomials. 

\begin{proposition}\label{prop:family}
Let $n=t t'$, for some $t,t' \in \mathbb{N}$, and let $s\in \mathbb{N}$ be such that $\gcd(s,t)=1$. Then 
\begin{equation}\label{eq:form} 
f(x)=\sum_{i=0}^{t'-1} a_i x^{q^{it+s}}\in\fqn[x] \end{equation}
is R-$q^t$-partially scattered if and only if $f(x)$ is invertible over $\fqn$.
\end{proposition}
\begin{proof}
Consider $\varphi(x)=x^{q^s}$, with $\gcd(s,t)=1$.
The polynomial $\varphi(x)$ is R-$q^t$-partially scattered, because of Proposition \ref{prop:mon}, and invertible over $\fqn$. 
Let $g_{\mathbf{a}}(x)=\sum_{i=0}^{t'-1} a_i x^{q^{it}}$. 
By Proposition \ref{prop:LZ} we have that $f(x)=(g_{\mathbf{a}}\circ \varphi)(x) $
is R-$q^t$-partially scattered if and only if $g_{\mathbf{a}}(x)$ is invertible over $\fqn$.
\end{proof}

\begin{corollary}\label{cor:bin}
Let $n=t t'$, for some $t,t' \in \mathbb{N}$, and let $s,k>0$ be such that $\gcd(s,t)=1$ and $kt+s<n$. Then $f(x)=x^{q^{kt+s}}+\alpha x^{q^s}\in\fqn[x]$ is R-$q^t$-partially scattered if and only if $\N_{q^n/q^{t\cdot\gcd(k,t^\prime)}}(-\alpha)\ne 1$.
In this case, $f(x)$ is exceptional R-$q^t$-partially scattered.
\end{corollary}

\begin{proof}
As $f(x)=(g_{\mathbf{a}}\circ \varphi)(x)$, where $g_{\mathbf{a}}(x)=x^{q^{kt}}+\alpha x$ and $\varphi(x)=x^{q^s}$, by  Proposition \ref{prop:family}, $f(x)$ is R-$q^t$-partially scattered if and only if $\N_{q^n/q^{t\cdot\gcd(k,t^\prime)}}(-\alpha)\ne 1$.
In this case we have $\N_{q^{mn}/q^{t\cdot\gcd(k,mt^\prime)}}(-\alpha)=\N_{q^{mn}/q^n}(\N_{q^{n}/q^{t\cdot\gcd(k,t^\prime)}}(-\alpha))\ne 1$ for any $m\in\mathbb{N}$ with $\gcd(m,k(q^n-1))=1$, and the claim follows.
\end{proof}

\begin{corollary}
Let $n=3t$, for some $t \in \mathbb{N}$, and let $s\in \mathbb{N}$ be such that $\gcd(s,t)=1$. Then $f(x)=x^{q^{2t+s}}+\beta x^{q^{t+s}}+\alpha x^{q^s}$ is R-$q^t$-partially scattered if and only if $\N_{q^{3t}/q^t}(\alpha)+\N_{q^{3t}/q^t}(\beta)-\mathrm{Tr}_{q^{3t}/q^t}(\alpha \beta^{q^t})+1\ne0$.
\end{corollary}
\begin{proof}
By Proposition \ref{prop:family}, $f(x)$ is R-$q^t$-partially scattered if and only if the $\F_{q^t}$-linearized polynomial $g(x)=\alpha x+\beta x^{q^{t}}+x^{q^{2t}}\in \fqn[x]$ is invertible.
This happens if and only if
\[ \det\left( 
\begin{array}{ccc}
\alpha & \beta & 1\\
1 & \alpha^{q^t} & \beta^{q^t}\\
\beta^{q^{2t}} & 1 & \alpha^{q^{2t}}
\end{array}
\right)\ne0, \]
from which the assertion follows.
\end{proof}

We are able to determine and count all the polynomials in \eqref{eq:form} which are R-$q^t$-partially scattered.
To this aim, we recall the following result.

\begin{theorem}\cite[Theorem 2.1 and Corollary 2.3]{polinv}\label{th:polinv}
Let $\alpha$ be any primitive element in $\fqn$. Every invertible $\fq$-linearized polynomial $f(x)$ over $\fqn$ (of $q$-degree smaller than $n$) has the following form:
\begin{equation}\label{eq:forma}
f(x)=\sum_{i=0}^{n-1}(\alpha_0+\alpha^{q^i}\alpha_1+\ldots+\alpha^{(n-1)q^i}\alpha_{n-1})x^{q^i} \in \fqn[x],
\end{equation}
where $\{\alpha_0,\ldots,\alpha_{n-1}\}$ is any $\fq$-basis of $\fqn$.
Conversely, every polynomial over $\mathbb{F}_{q^n}$ of the form \eqref{eq:forma} is invertible.
In particular, the number of invertible $\fq$-linearized polynomials over $\fqn$ (of $q$-degree smaller than $n$) is
\begin{equation}\label{eq:num} 
(q^n-1)\cdot (q^n-q)\cdot \ldots \cdot (q^n-q^{n-1}). 
\end{equation}
\end{theorem}

This allows us to prove the following result.

\begin{theorem}
Let $n=t t'$, for some $t,t' \in \mathbb{N}$, and let $s\in \mathbb{N}$ be such that $\gcd(s,t)=1$. The  R-$q^t$-partially scattered linearized polynomials of the form \eqref{eq:form} are those of the shape
\[ f(x)=\sum_{i=0}^{t'-1}(\alpha_0+\alpha^{q^{it}}\alpha_1+\ldots+\alpha^{(t'-1)q^{it}}\alpha_{t'-1})x^{q^{it+s}}, \]
where $\alpha$ is any primitive element in $\fqn$ and $\{\alpha_0,\ldots,\alpha_{t'-1}\}$ is any $\F_{q^t}$-basis of $\fqn$.
In particular the number of R-$q^t$-partially scattered polynomials of the form \eqref{eq:form} is
\[ (q^n-1)\cdot(q^n-q^t)\cdot \ldots \cdot (q^n-q^{n-t}). \]
\end{theorem}
\begin{proof}
By Proposition \ref{prop:family}, the problem of determining which linearized polynomials of Form \eqref{eq:form} are R-$q^t$-partially scattered may be translated in determining the invertible $\F_{q^t}$-linearized polynomials $g_{\mathbf{a}}(x)$ over $\fqn$. 
By Theorem \ref{th:polinv}, $g_{\mathbf{a}}(x)$ is invertible if and only if 
\[ g_{\mathbf{a}}(x)=\sum_{i=0}^{n-1}(\alpha_0+\alpha^{q^i}\alpha_1+\ldots+\alpha^{(n-1)q^i}\alpha_{n-1})x^{q^i}, \]
for some primitive element $\alpha$ in $\fqn$ and $\F_{q^t}$-basis $\{\alpha_0,\ldots,\alpha_{t'-1}\}$ of $\fqn$.
Also, two linearized polynomials $f(x)=\sum_{i=0}^{t'-1} a_i x^{q^{s+it}}$ and $h(x)=\sum_{i=0}^{t'-1} b_i x^{q^{s+it}}$ of the form \eqref{eq:form} coincide if and only if $a_i=b_i$ for every $i \in \{0,\ldots,t'-1\}$, so that the last part of the assertion follows by \eqref{eq:num}.
\end{proof}

For $\fq$-linearized binomials of the form \eqref{eq:form} with $t'=2$, we may use the results in \cite{PZZ} to get examples of linearized binomials which are R-$q^t$-partially scattered but not scattered.

\begin{proposition}
Let $n=2t$, for some $t \in \mathbb{N}$ and let $s\in \mathbb{N}$ be such that $\gcd(s,t)=1$. If 
\[ n\geq \left\{ \begin{array}{ll} 
4s+2 & \text{if}\,\,q=3\,\,\text{and}\,\, s>1,\,\,\text{or}\,\, q=2\,\,\text{and}\,\, s>2,\\
4s+1 & \text{otherwise},
\end{array} \right. \]
then the linearized polynomial
\[ f(x)=\delta x^{q^s}+x^{q^{t+s}}\in \fqn[x], \]
with $\N_{q^{2t}/q^t}(\delta)\ne 1$, is R-$q^t$-scattered but not scattered.
\end{proposition}
\begin{proof}
By \cite[Theorem 4.1]{PZZ}, it follows the existence of $m \in \fqn^*$ such that
\[ \dim_{\F_q}(\ker(f(x)-mx))=2, \]
which implies the existence of $y,z \in \fqn^*$ such that $y/z \notin \fq$ and 
\[ \frac{f(y)}y=m=\frac{f(z)}z, \]
that is $f(x)$ is not scattered.
The assertion then follows by Corollary \ref{cor:bin}.
\end{proof}

\begin{remark}
The family presented in Proposition \ref{prop:family} contains the examples of R-$q^t$-partially scattered given in \cite{LZ}:
\begin{itemize}
    \item $f(x)=\delta x^{q^s}+x^{q^{t+s}} \in \F_{q^{2t}}[x]$ with $\gcd(2t,s)=1$ and $\N_{q^{2t}/q^t}(\delta)\ne 1$, see \cite[Proposition 2.12]{LZ};
    \item $f(x)=x^{q^s}+x^{q^{t+s}}+\delta x^{q^{2t+s}}\in \F_{q^{3t}}$ with $\gcd(3t,s)=1$ and $\mathrm{Tr}_{q^{3t}/q^t}(\delta)-\N_{q^{3t}/q^t}(\delta)\ne 2$, see \cite[Proposition 2.16]{LZ};
    \item $f(x)=x^{q^s}+x^{q^{t+s}}+x^{q^k}-x^{q^{t+k}}\in \F_{q^{2t}}[x]$ with $q$ odd, $\gcd(s,2t)=\gcd(k,2t)=1$ and $0\leq s,k\leq 2t-1$, see \cite[Proposition 2.18]{LZ}. This family generalizes the example in \cite{LZ2}, which for $2t=6$ may be rewritten as in \eqref{eq:form}, see \cite[Proposition 3.9]{BZZ} and see also \cite{ZZ}.
\end{itemize}
\end{remark}

\begin{remark}
The family of R-$q^t$-partially scattered polynomials introduced in this section is closed under the adjoint operation. Indeed, let \[ f(x)=a_0 x^{q^s}+a_1 x^{q^{t+s}}+\ldots+a_{t'-1}x^{q^{(t'-1)t+s}} \in \fqn[x] \]
be such that $g_{\mathbf{a}}(x)=a_0x+a_1x^{q^t}+\ldots+a_{t'-1}x^{q^{(t'-1)t}}$ is invertible.
Then 
\[ \hat{f}(x)=a_0^{q^{n-s}}x^{q^{n-s}}+a_1^{q^{t(t'-1)-s}}x^{q^{t(t'-1)-s}}+\ldots+a_{t'-1}^{q^{t-s}}x^{q^{t-s}}, \]
see \eqref{eq:adj}.
Define $s'=t-s$, then $\gcd(s',t)=1$ and
\[ \hat{f}(x)=a_{t'-1}^{q^{s'}} x^{q^{s'}}+a_{t'-2}^{q^{t+s'}}x^{q^{t+s'}}+\ldots+a_0^{q^{t(t'-1)+s'}}x^{q^{t(t'-1)+s'}}, \]
hence $\hat{f}(x)$ is of the form \eqref{eq:form}, because of Proposition \ref{prop:adj}.
\end{remark}

\section{A geometric description}\label{sec:pseudo}

A point set $L$ of $\Lambda=\PG(V,\F_{q^n})\allowbreak=\PG(r-1,q^n)$ is said to be an \emph{$\F_q$-linear set} of $\Lambda$ of rank $k$ if it is defined by the non-zero vectors of a $k$-dimensional $\F_q$-vector subspace $U$ of $V$, i.e.
\[L=L_U:=\{\la {\bf u} \ra_{\mathbb{F}_{q^n}} : {\bf u}\in U\setminus \{{\bf 0} \}\}.\]
For any subspace $S=\PG(Z,\mathbb{F}_{q^n})$ of $\Lambda$, the \emph{weight} of $S$ in $L_U$ is defined as $w_{L_U}(S)=\dim_{\mathbb{F}_q}(U\cap Z)$.
We also recall that two linear sets $L_U$ and $L_W$ of $\PG(r-1,q^n)$ are said to be \emph{$\mathrm{P\Gamma L}$-equivalent} (or simply \emph{equivalent}) if there is an element $\varphi$ in $\mathrm{P\Gamma L}(r,q^n)$ such that $L_U^{\varphi} = L_W$.
An important family of linear set is given by the so called \emph{scattered linear sets}, that is linear sets defined by scattered subspaces.
Clearly, the rank of a scattered linear set is bounded above by $\frac{rn}2$.
For further details on linear sets see \cite{LVdV2015,Polverino}.

In \cite{LMPT:14}, generalizing \cite{LMPT2011,LVdV,MPT},  a class of scattered linear sets of maximum rank in $\mathrm{PG}(t'-1,q^t)$ was presented.
Let $L_U$ be a scattered $\mathbb{F}_q$-linear set of $\Lambda=\mathrm{PG}(2t'-1,q^t)$ of rank $tt'$, with $t,t'\geq 2$.
We say that $L_U$ is of \emph{pseudoregulus type} if 
\begin{itemize}
    \item there exist $m=\frac{q^{tt'}-1}{q^t-1}$ pairwise disjoint lines  $s_1,\ldots,s_m$ of $\Lambda$ such that $w_{L_U}(s_i)=t$ for every $i \in \{1,\ldots,m\}$;
    \item there exist exactly two $(t'-1)$-dimensional subspaces $T_1$ and $T_2$ of $\Lambda$ disjoint from $L_U$ and such that $T_j\cap s_i\ne \emptyset$, for every $j \in \{1,2\}$ and $i \in \{1,\ldots,m\}$.
\end{itemize}

The set $\mathcal{P}_{L_U}=\{s_1,\ldots,s_m\}$ is called \emph{pseudoregulus} of $\Lambda$ associated with $L_U$, and $T_1$ and $T_2$ are called \emph{transversal spaces of} $\mathcal{P}_{L_U}$.
See \cite{NPZZ} for a generalization.
In \cite{LMPT:14}, an algebraic characterization of linear sets of pseudoregulus type has been provided.

\begin{theorem}\cite[Theorem 3.5]{LMPT:14}\label{th:pseudo}
Let $T_1=\mathrm{PG}(U_1,\mathbb{F}_{q^t})$ and $T_2=\mathrm{PG}(U_2,\mathbb{F}_{q^t})$ be two disjoint $(t'-1)$-subspaces of $\Lambda=\mathrm{PG}(2t'-1,q^t)$ and let $\phi_f$ be the strictly semilinear collineation between $T_1$ and $T_2$ defined by an invertible $\mathbb{F}_{q^{t}}$-semilinear map $f$ with companion automorphism $\sigma \in \mathrm{Aut}(\mathbb{F}_{q^t})$ such that $\mathrm{Fix}(\sigma)=\mathbb{F}_q$.
Then for each $\rho \in \mathbb{F}_{q^t}^*$, the set
\[ L_{\rho,f}=\{\langle \mathbf{u}+\rho f(\mathbf{u}) \rangle_{\mathbb{F}_{q^t}} \colon \mathbf{u}\in U_1\setminus\{\mathbf{0}\}\} \]
is an $\mathbb{F}_q$-linear set of pseudoregulus type of $\Lambda$ whose associated pseudoregulus is $\mathcal{P}_{L_{\rho,f}}=\{\langle P,P^{\phi_f} \rangle \colon P \in T_1 \}$ and whose transversal spaces are $T_1$ and $T_2$.
Conversely, each $\F_q$-linear set of pseudoregulus type of $\Lambda$ can be obtained as above.
\end{theorem}

Now, let $\mathbb{V}=\mathbb{F}_{q^{tt'}}\times\mathbb{F}_{q^{tt'}}$, which can be seen simultaneously as both a $2$-dimensional $\mathbb{F}_{q^{tt^\prime}}$-vector space and a $2t^\prime$-dimensional $\mathbb{F}_{q^t}$-vector space.
Consider $U_1=\{(x,0) \colon x \in \F_{q^{tt'}}\}=V(t',q^t)$ and $U_2=\{ (0,y)  \colon y \in \F_{q^{tt'}}\}=V(t^\prime,q^t)$, and let $T_i=\PG(U_i,\mathbb{F}_{q^t})=\PG(t^\prime-1,q^t)$.
By Theorem \ref{th:pseudo}, the $\F_q$-linear sets of pseudoregulus type in $\Lambda=\PG(\mathbb{V},\mathbb{F}_{q^t})=\PG(2t'-1,q^t)$ with transversal spaces $T_1$ and $T_2$ are exactly those of the form
\[ L_f=\{ \langle (x,f(x)) \rangle_{\mathbb{F}_{q^t}} \colon x \in \F_{q^{tt'}}^* \}, \]
where $f(x)$ is a strictly $\F_{q^t}$-semilinear function of $\F_{q^{tt'}}$ with companion automorphism $\sigma\in \mathrm{Aut}(\F_{q^{t}})$ with $\mathrm{Fix}(\sigma)=\F_q$, that is
\[ f(x)=\sum_{i=0}^{t'-1} a_i x^{\sigma q^{it}} \in \mathbb{F}_{q^{tt'}}[x], \]
where $\sigma\colon x \in \F_{q^{tt'}}\mapsto x^{q^s} \in \F_{q^{tt'}}$ with $\gcd(s,t)=1$ and $f(x)$ is invertible.

Therefore, using Proposition \ref{prop:family}, the following holds.

\begin{corollary}
Let $n=t t'$, for some $t,t'\in\mathbb{N}$, and let 
\[ f(x)=\sum_{i=0}^{t'-1} a_i x^{\sigma q^{it}} \in \mathbb{F}_{q^{n}}[x], \]
where $\sigma\colon x \in \F_{q^{n}}\mapsto x^{q^s} \in \F_{q^{n}}$ with $\gcd(s,t)=1$ and $f(x)$.
Then $f(x)$ is R-$q^t$-partially scattered if and only if
\[ L_f=\{ \langle (x,f(x)) \rangle_{\mathbb{F}_{q^t}} \colon x \in \F_{q^{tt'}}^* \} \subseteq \Lambda=\mathrm{PG}(2t'-1,q^t) \]
is an $\fq$-linear set of pseudoregulus type.
\end{corollary}

\begin{remark}
If $f(x)$ is as in \eqref{eq:form}, by Theorem \ref{th:pseudo}, $L_f$ is a scattered $\F_q$-linear set of rank $tt'$ in $\mathrm{PG}(2t'-1,q^t)$, and hence, by Theorem \ref{th:scatteredbig}, $f(x)$ is R-$q^t$-partially scattered.
This in an alternative proof of one implication in Proposition \ref{prop:family}.
\end{remark}


\section{Equivalence issue}\label{sec:equiv}

We start by defining a natural equivalence between two linearized polynomials. 
Let $f(x)$ and $g(x)$ be two $\F_q$-linearized polynomials over $\fqn$ and consider the two $\fq$-subspaces
\[U_f=\{ (x,f(x)) \colon x \in \fqn \}\qquad \mbox{and} \qquad U_g=\{ (x,g(x)) \colon x \in \fqn \}\]
of $\fqn\times\fqn$.
We say that $f(x)$ and $g(x)$ are \emph{equivalent} if there exists $\varphi \in \mathrm{\Gamma L}(2,q^n)$ such that $U_f^\varphi=U_g$, that is, there exist $A\in \mathrm{GL}(2,q^n)$ and $\sigma \in \mathrm{Aut}(\fqn)$ with the property that for each $x \in \fqn$ there exists $y \in \fqn$ satisfying
\[ A \left( \begin{array}{cc} x^\sigma\\ f(x)^\sigma \end{array} \right)=\left( \begin{array}{cc} y\\ g(y) \end{array} \right), \]
see \cite[Section 1]{BMZZ} and \cite[Section 1]{CsMPq5}.

This definition of equivalence preserves the property of being R-$q^t$-partially scattered.

\begin{proposition}
Let $f(x)$ and $g(x)$ be two equivalent $\fq$-linearized polynomials of $\F_{q^n}[x]$.
If $f(x)$ is R-$q^t$-partially scattered, then $g(x)$ is R-$q^t$-partially scattered.
\end{proposition}
\begin{proof}
Since $U_f$ and $U_g$ are $\Gamma\mathrm{L}(2,q^n)$-equivalent, $U_f$ and $U_g$ are also $\Gamma\mathrm{L}(2t',q^t)$-equivalent.
If $f(x)$ is R-$q^t$-partially scattered, Theorem \ref{th:scatteredbig} implies that $U_f$ is a scattered subspace of $\F_{q^n}\times\F_{q^n}=V(2t',q^t)$.
Therefore $U_g$ is a scattered subspace of $\F_{q^n}\times\F_{q^n}=V(2t',q^t)$, and hence  $g(x)$ is R-$q^t$-partially scattered by Theorem \ref{th:scatteredbig}.
\end{proof}

In order to establish whether two $\fq$-linearized polynomial are equivalent or not the following definition may help.
Let $f(x)$ be an $\F_q$-linearized polynomial, then we define the \emph{linear automorphism group} of $f(x)$ as follows
\[ \mathcal{G}(f)=\left\{ A \in \mathrm{GL}(2,q^n) \mid \forall x \in \fqn \ \ \exists y \in \fqn \colon A \left( \begin{array}{cc} x\\ f(x) \end{array} \right)=\left( \begin{array}{cc} y\\ f(y) \end{array} \right) \right\}.\]

The following lemma clearly holds.

\begin{lemma}\label{lemma:aut}
Let $f(x)$ and $g(x)$ be two $\F_q$-linearized polynomials over $\fqn$. If $f(x)$ and $g(x)$ are equivalent then $\mathcal{G}(f)$ and $\mathcal{G}(g)$ are isomorphic.
In particular, if $|\mathcal{G}(f)|\ne |\mathcal{G}(g)|$, then $f(x)$ and $g(x)$ are not equivalent.
\end{lemma}

The linear automorphism group has been determined in the following cases:
\begin{itemize}
\item $f(x)=x^{q^s}\in \fqn[x]$ with $\gcd(s,n)=1$, then $|\mathcal{G}(f)|=q^n-1$, see \cite[Section 6]{CMPZ};
\item $f(x)=\delta x^{q^s}+x^{q^{n(s-1)}}\in \fqn[x]$ with $\gcd(s,n)=1$ and $n\geq 4$, then $|\mathcal{G}(f)|=q^2-1$ if $n$ is even  and $|\mathcal{G}(f)|=q-1$ if $n$ is odd, see \cite[Section 6]{CMPZ};
\item $f(x)=\delta x^{q^s}+x^{q^{s+n/2}}\in \fqn[x]$ with $n$ even and $\gcd(s,n)=1$, then  $|\mathcal{G}(f)|=q^{n/2}-1$, see \cite[Corollary 5.2]{CMPZ};
\item $f(x)=x^q+x^{q^3}+\delta x^{q^5}\in \F_{q^6}[x]$ with $q$ odd and $\delta^2+\delta=1$, then $|\mathcal{G}(f)|=q^{2}-1$, see \cite[Proposition 5.2]{CsMZ2018} and \cite[Section 4.4.]{MMZ}.
\end{itemize}

Now we determine the linear automorphism group of the binomial over $\fqn$ given in Corollary \ref{cor:bin} .

\begin{proposition}\label{prop:groupbin}
Let $n=t t'$, for some $t,t' \in \mathbb{N}$ with $t>1$, and let $s,k>0$ be such that $\gcd(s,t)=1$ and $kt+s<n$. Let $\alpha \in \fqn^*$ be such that $\N_{q^n/q^{t\cdot\gcd(k,t^\prime)}}(-\alpha)\ne 1$. Denote by $G$ the group $$\left\{\left(\begin{array}{cc} 
a & 0\\
0 & a^{q^s}
\end{array}\right) \colon a\in \F_{q^{t\cdot\gcd(k,t')}}^*\right\}.$$
\begin{enumerate}
    \item $\mathcal{G}(x^{q^{kt+s}}+\alpha x^{q^s})\supset G$.
    \item Assume also that $t' \ne 2k$.  Then 
$\mathcal{G}(x^{q^{kt+s}}+\alpha x^{q^s})=G$.
\end{enumerate}
In particular, $|\mathcal{G}(x^{q^{kt+s}}+\alpha x^{q^s})|\geq q^{t\cdot\gcd(k,t')}-1$.
\end{proposition}
\begin{proof}
The first part is straightforward. Now we show the second part.
Let $\left(\begin{array}{cc} 
a & b\\
c & d
\end{array}\right)\in \mathrm{GL}(2,q^n)$ have the property that for each $x \in \fqn$ there exists $y \in \fqn$ such that
\[ \left(\begin{array}{cc} 
a & b\\
c & d
\end{array}\right) \left(\begin{array}{cc} 
x\\
x^{q^{kt+s}}+\alpha x^{q^s}
\end{array}\right)=\left(\begin{array}{cc} 
y\\
y^{q^{kt+s}}+\alpha y^{q^s}
\end{array}\right). \]
Then 
\[ cx+d(x^{q^{tk+s}}+\alpha x^{q^s})=\]
\[ a^{q^{kt+s}}x^{q^{kt+s}}+b^{q^{kt+s}}(x^{q^{2(tk+s)}}+\alpha^{q^{kt+s}}x^{q^{kt+2s}}) +\alpha[a^{q^s}x^{q^s}+b^{q^s}(x^{q^{kt+2s}}+\alpha^{q^s}x^{q^{2s}})]. \]
The powers of $x$ in the above equality are $q^0,q^s,q^{kt+s},q^{2(kt+s)},q^{kt+2s},q^{2s}$.
Consider 
\[
\mathcal{T}= \{0\  \mod{n},\;s\ \mod{n},\;kt+s\ \mod{n},\;2(kt+s)\ \mod{n},\;kt+2s \ \mod{n},\;2s\ \mod{n}\}.
\]

Two elements in $\mathcal{T}$ may coincide if and only if
\begin{itemize}
    \item $2(tk+s) \equiv 0 \pmod{n}$ (which happens if and only if $t=2$ and $t'= 2k+s$);
    \item $tk+2s \equiv 0 \pmod{n}$ (which happens if and only if $t=2$ and $t'= k+s$);
    \item $2s \equiv 0 \pmod{n}$ (which happens if and only if $t=2$ and $t'= s$).
    \item $2(tk+s) \equiv 2s \pmod{n}$ (which happens if and only if $t'= 2k$, a contradiction to the assumptions).
\end{itemize}

If $t\neq2$, or $t=2$ and $t'\notin\{2k+s,k+s,s\}$, then $\mathcal{T}$ has size $6$.
By looking at the coefficients of $x$ and of $x^{q^{2(kt+s)}}$, we get $b=c=0$.
From the coefficients of $x^{q^{s}}$ and $x^{q^{kt+s}}$ we get $d=a^{q^s}$ and $d=a^{q^{tk+s}}$, whence  $a \in \F_{q^{t\cdot\gcd(k,t')}}$ and the statement is proved.

If $t=2$ and $t^\prime$ equals at least one among $s$, $k+s$, and $2k+s$, then $\mathcal{T}$ has size  $5$. In fact, at most one among $t'= 2k+s$, $t'= k+s$ and $t'= s$ can occur. We distinguish three cases.
\begin{enumerate}
    \item $t^\prime= 2k+s$. By looking at the coefficient of $x^{q^{2s}}$, we get $b=0$.
    \item $t^\prime= k+s$. By looking at the coefficient of $x^{q^{2s}}$, we get $b=0$.
    \item $t^\prime= s$. By looking at the coefficient of $x^{q^{2(tk+s)}}$, we get $b=0$.
\end{enumerate}
Now, by looking at the coefficient of $x$, $x^{q^s}$ and $x^{q^{kt+s}}$, we get  $c=0$, $d=a^{q^s}$, and $a\in\mathbb{F}_{q^{t\cdot\gcd(k,t^\prime)}}$.
\end{proof}

We can determine a subgroup of the linear automorphism group of a polynomial of the form \eqref{eq:form}. 

\begin{proposition}\label{prop:groupgen}
Let $n=t t'$, for some $t,t' \in \mathbb{N}$, and let $s\in \mathbb{N}$ be such that $\gcd(s,t)=1$. For a polynomial  
\begin{equation*}
f(x)=\sum_{i=0}^{t'-1} a_i x^{q^{it+s}}\in \fqn[x]. 
\end{equation*}
we have that
\[ \mathcal{G}(f(x))\supseteq \left\{\left(\begin{array}{cc} 
a & 0\\
0 & a^{q^s}
\end{array}\right) \colon a\in \F_{q^{t}}^*\right\}.\]
In particular, $|\mathcal{G}(f(x))|\geq q^t-1$.
\end{proposition}

In \cite[Remark 7.2]{PZ2019}, it was pointed out that when $t=2$ and $n=2t'$, an LP-polynomial can be written in the form \eqref{eq:form}.
Here we generalize it to a larger family.

\begin{corollary}
Let $n=t t'$, for some $t,t' \in \mathbb{N}$ with $t>1$, and let $s,k\in \mathbb{N}$ be such that $\gcd(s,t)=1$ and $kt+s<n$. Let $\alpha \in \fqn$ be such that $\N_{q^n/q^{\gcd(kt,n)}}(-\alpha)\ne 1$.
If $f(x)=x^{q^{kt+s}}+\alpha x^{q^s}$ is equivalent to an LP-polynomial, then $t=2$ and $n=2t'$.
Conversely, if $t=2$, $n=2t^\prime$, and $\gcd(k,2t^\prime)=1$, then $f(x)$ is equivalent to an LP-polynomial.
\end{corollary}
\begin{proof}
If $f(x)$ is equivalent to an LP-polynomial, then by Lemma \ref{lemma:aut} and Proposition \ref{prop:groupbin} we have in particular that $t=2$ and $\gcd(k,t^\prime)=1$. The converse follows by \cite[Remark 7.2]{PZ2019}.
\end{proof}

In the case $n=6$ and $t=3$, we determine the number of inequivalent linearized binomials of the form \eqref{eq:form} which are either scattered, or R-$q^t$-partially scattered but not scattered.

\begin{proposition}\label{prop:equivbin}
Let $q=p^e$, for some prime $p$ and positive integer  $e$. 
The number of inequivalent scattered $\fq$-linearized binomials in the set $\Delta := \{\delta x^{q^s}+x^{q^{s+3}} : s\in \{1,2,4,5\}, \delta\in\mathbb{F}_{q^6}\}$, is
\[ \Gamma= \frac{\lvert (q^2+q+1)(q-2)/2 \rvert}{3e}. \]
The  number of inequivalent $f(x)\in\Delta$ which are R-$q^t$-scattered but not scattered is 
\[ q^3-1-\Gamma. \]
In particular, there is only one equivalence class of binomials of the form \eqref{eq:form} which are neither scattered nor R-$q^t$-partially scattered.
\end{proposition}
\begin{proof}
By \cite[Introduction]{BCM} (see also \cite[Proposition 5.1]{CMPZ}), it is enough to study the case $s=1$, then by applying \cite[Theorem 1.3]{BCM} we get the first part of the assertion. 
The second part follows from the fact that $f(x)=\delta x^{q}+x^{q^{4}}$ is R-$q^t$-partially scattered if and only if $\N_{q^6/q^3}(\delta)\ne 1$, see \cite[Proposition 2.12]{LZ}.
\end{proof}

Regarding the R-$q^t$-partially scattered quadrinomials introduced in \cite[Proposition 2.18]{LZ} we can prove the following.

\begin{proposition}\label{Prop:prima}
Let $t,s,k \in \mathbb{N}$ be such that $\gcd(s,2t)=\gcd(k,2t)=1$ and $t\geq 2$, and let $n=2t$. Suppose that  $\{ 0,s,t+s,k,t+k,2s,t+2s,k+s,t+k+s,2k,t+2k \}$ has size $11$ as a subset of $\mathbb{Z}/n\mathbb{Z}$.

If $q$ is odd, then 
\[ \mathcal{G}(x^{q^s}+x^{q^{t+s}}+x^{q^k}-x^{q^{t+k}})=\left\{\left(\begin{array}{cc} 
a & 0\\
0 & a^{q^s}
\end{array}\right) \colon a\in \F_{q^{\gcd(t,\lvert k-s\rvert)}}^*\right\};\]
In particular, $|\mathcal{G}(x^{q^s}+x^{q^{t+s}}+x^{q^k}-x^{q^{t+k}})|=q^{\gcd(t,\lvert k-s\rvert)}-1$.

If $q$ is even, then
\[ \mathcal{G}(x^{q^s}+x^{q^{t+s}}+x^{q^k}+x^{q^{t+k}})=\left\{\left(\begin{array}{cc} 
a & b\\
0 & a^{q^s}
\end{array}\right) \colon a\in \F_{q^{\gcd(t,\lvert k-s\rvert)}}^*,\,b^{q^t}=b\right\};\]
In particular, $|\mathcal{G}(x^{q^s}+x^{q^{t+s}}+x^{q^k}+x^{q^{t+k}})|=q^t(q^{\gcd(t,\lvert k-s\rvert)}-1)$.
\end{proposition}
\begin{proof}
Let $\left(\begin{array}{cc} 
a & b\\
c & d
\end{array}\right)\in \mathrm{GL}(2,q^n)$ with the property that for each $x \in \fqn$ there exists $y \in \fqn$ such that
\[ \left(\begin{array}{cc} 
a & b\\
c & d
\end{array}\right) \left(\begin{array}{cc} 
x\\
x^{q^s}+x^{q^{t+s}}+x^{q^k}-x^{q^{t+k}}
\end{array}\right)=\left(\begin{array}{cc} 
y\\
y^{q^s}+y^{q^{t+s}}+y^{q^k}-y^{q^{t+k}}
\end{array}\right). \]
Then
\begin{eqnarray} cx+d(x^{q^s}+x^{q^{t+s}}+x^{q^k}-x^{q^{t+k}})=a^{q^{s}}x^{q^{s}}+b^{q^{s}}(x^{q^{2s}}+x^{q^{t+2s}}+x^{q^{k+s}}-x^{q^{t+k+s}})\nonumber\\
-a^{q^{t+k}}x^{q^{t+k}}+a^{q^{t+s}}x^{q^{t+s}}+b^{q^{t+s}}(x^{q^{t+2s}}+x^{q^{2s}}+x^{q^{t+k+s}}-x^{q^{k+s}})+a^{q^k}x^{q^k} \nonumber\\
+b^{q^k}(x^{q^{k+s}}+x^{q^{t+k+s}}+x^{q^{2k}}-x^{q^{t+2k}})-b^{q^{t+k}}(x^{q^{t+k+s}}+x^{q^{k+s}}+x^{q^{t+2k}}-x^{q^{2k}}).\label{eq:polidentity}\end{eqnarray}
From the coefficient of $x$ we get $c=0$ and from the coefficients of $x^{q^s}$, $x^{q^{t+s}}$, $x^{q^k}$, and $x^{q^{t+k}}$ we obtain $d=a^{q^s}$ and $a\in \F_{q^{\gcd(t,\lvert k-s\rvert)}}$.

If $q$ is odd, then by the coefficients of $x^{q^{2s}}$, $x^{q^{k+s}}$ and $x^{q^{t+k+s}}$ one gets $b=0$.
If $q$ is even, then by the coefficient of $x^{q^{2s}}$ one gets $b^{q^t}=b$. Together with $a\in\mathbb{F}_{q^{\gcd(t,|k-s|)}}^*$, $c=0$, and $d=a^{q^s}$, this is enough to satisfy the polynomial identity \eqref{eq:polidentity}.
\end{proof}

\begin{corollary}\label{cor:LZ1}
Let $t,s,k \in \mathbb{N}$ be such that $\gcd(s,2t)=\gcd(k,2t)=1$ and $t\geq 2$, and let $n=2t$.
Assume also that $2s<t$, $2k<t$, and $s\ne k$.
If $q$ is odd, then 
\[ \mathcal{G}(x^{q^s}+x^{q^{t+s}}+x^{q^k}-x^{q^{t+k}})=\left\{\left(\begin{array}{cc} 
a & 0\\
0 & a^{q^s}
\end{array}\right) \colon a\in \F_{q^{\gcd(t,\lvert k-s\rvert)}}^*\right\};\]
and if $q$ is even
\[ \mathcal{G}(x^{q^s}+x^{q^{t+s}}+x^{q^k}+x^{q^{t+k}})=\left\{\left(\begin{array}{cc} 
a & b\\
0 & a^{q^s}
\end{array}\right) \colon a\in \F_{q^{\gcd(t,\lvert k-s\rvert)}}^*,\,b^{q^t}=b\right\}.\]
In particular, $x^{q^s}+x^{q^{t+s}}+x^{q^k}-x^{q^{t+k}}$ is not equivalent to any polynomial of the form \eqref{eq:form}.
\end{corollary}
\begin{proof}
Since $s$ and $k$ are odd, $s<t/2$, and  $s\ne k$, $\{ 0,s,t+s,k,t+k,2s,t+2s,k+s,t+k+s,2k,t+2k \}$ has size 11 and the group $\mathcal{G}(x^{q^s}+x^{q^{t+s}}+x^{q^k}+x^{q^{t+k}})$ follows from  Proposition \ref{Prop:prima}.
Since its order is not divisible by $q^t-1$, the claim follows by Proposition \ref{prop:groupgen}.
\end{proof}

Regarding the scattered polynomials presented in \cite{LZ2} the linear automorphism group can be completely determined.

\begin{corollary}\label{cor:LZ2}
Let $q$ be an odd prime power. Let $t,k,n \in \mathbb{N}$ with $n=2t$ be such that either $k=1$ and $t\geq5$, or $k>1$, $\gcd(k,2t)=1$, and $t> 2k$.
Then 
\[ \mathcal{G}(x^{q^{t-k}}+x^{q^{2t-k}}+x^{q^k}-x^{q^{t+k}})=\left\{\left(\begin{array}{cc} 
a & 0\\
0 & a^{q}
\end{array}\right) \colon a\in \F_{q^{2}}^*\right\},\]
if $t$ is even,
and 
\[ \mathcal{G}(x^{q^{t-k}}+x^{q^{2t-k}}+x^{q^k}-x^{q^{t+k}})=\left\{\left(\begin{array}{cc} 
a & b\\
-4b & a
\end{array}\right) \colon a\in \F_{q},b^q+b=0, a^2+4b^2\ne0 \right\},\]
if $t$ is odd.
\end{corollary}
\begin{proof}
The polynomial identity \eqref{eq:polidentity} reads
\[ cx+d(x^{q^{t-k}}+x^{q^{2t-k}}+x^{q^k}-x^{q^{t+k}})=a^{q^{t-k}}x^{q^{t-k}}+b^{q^{t-k}}(x^{q^{2t-2k}}+x^{q^{t-2k}}+x^{q^{t}}-x)\]
\[+a^{q^{2t-k}}x^{q^{2t-k}}+b^{q^{2t-k}}(x^{q^{t-2k}}+x^{q^{2t-2k}}+x-x^{q^{t}})+a^{q^k}x^{q^k}+b^{q^k}(x^{q^{t}}+x+x^{q^{2k}}-x^{q^{t+2k}}) \]
\[ -a^{q^{t+k}}x^{q^{t+k}}-b^{q^{t+k}}(x+x^{q^{t}}+x^{q^{t+2k}}-x^{q^{2k}}). \]
By our assumptions on $t$ and $k$ we have that the set 
\[\mathcal{T}=\{0, t-k, 2t-k, k, t+k, 2t-2k, t-2k, t, 2k, t+2k\}\subseteq \mathbb{Z}/n\mathbb{Z}\] 
has size $10$.
So, $a,b,c,d$ have to satisfy the following system
\begin{equation}\label{eq:condquadrin}
\left\{
\begin{array}{lllllllllll}
c=-b^{q^{t-k}}+b^{q^{2t-k}}+b^{q^k}-b^{q^{t+k}},\\
d=a^{q^{t-k}},\\
d=a^{q^{2t-k}},\\
d=a^{q^k},\\
d=a^{q^{t+k}},\\
b^{q^{t-k}}+b^{q^{2t-k}}=0,\\
b^{q^{t-k}}-b^{q^{2t-k}}+b^{q^k}-b^{q^{t+k}}=0,\\
\end{array}
\right.
\end{equation}
By the last equation we have $b^{q^t}+b=0$, which replaced in the seventh equation of \eqref{eq:condquadrin} implies $b^{q^{t+2k}}+b=0$, that is $b^{q^{2k}}=b$.
Since $\gcd(k,2t)=1$, this implies that $b \in \F_{q^2}$.
If $t$ is even, then $b^{q^t}+b=0$ implies $b=0$ and $c=0$, because of the first equation in \eqref{eq:condquadrin}.
When $t$ is odd, we have $b^q=-b$, $c=-4b$. Hence,  all equations of System \eqref{eq:condquadrin} but the second, third, and fourth ones  are satisfied.
The second, third, and fourth equations of \eqref{eq:condquadrin} together imply $d=a^{q^k}$ and $a \in \F_{q^{\gcd(t, t-2k)}}$.
The assertion then follows noting that $\gcd(t,t-2k)=2$ if $t$ is even and $\gcd(t,t-2k)=1$ if $t$ is odd.
\end{proof}


\begin{remark}
In \cite[Section 4]{LMTZ}, when $t\geq 5$ and $q$ is odd, the authors find the same results as in Corollaries \ref{cor:LZ1} and \ref{cor:LZ2}, but with a different approach. 
\end{remark}

\section{A weaker equivalence}\label{sec:weaker}

In this section we present a weaker definition of equivalence between two linearized polynomials of $\F_{q^n}[x]$ which preserves the property of being R-$q^t$-partially scattered.
Let $f(x)$ and $g(x)$ be two $\fq$-linearized polynomials in $\F_{q^n}[x]$.
We say that $f(x)$ and $g(x)$ are \emph{weakly equivalent} if there exists $\varphi \in \mathrm{\Gamma L}(2t',q^t)$ such that $U_f^\varphi=U_g$.

Clearly $\Gamma\mathrm{L}(2,q^{n})$ is a subgroup of $\Gamma\mathrm{L}(2t',q^{t})$ (as $\Gamma\mathrm{L}(2,q^{n})$ is the stabiliser of a Desarsguesian spread of $V(2t',q^t)$ in $\Gamma\mathrm{L}(2t',q^{t})$), and hence two equivalent $\fq$-linearized polynomials are also weakly equivalent. 

\begin{proposition}
Let $f(x)$ and $g(x)$ be two weakly equivalent $\fq$-linearized polynomials in $\F_{q^n}[x]$.
If $f(x)$ is R-$q^t$-partially scattered, then also $g(x)$ is R-$q^t$-partially scattered.
\end{proposition}
\begin{proof}
By hypothesis $U_f$ and $U_g$ are $\Gamma\mathrm{L}(2t',q^t)$-equivalent and $U_f$ is a scattered $\fq$-subspace of $V(2t',q^t)$, so that also $U_g$ is a scattered $\fq$-subspace of $V(2t',q^t)$ and hence $g(x)$ is R-$q^t$-partially scattered because of Theorem \ref{th:scatteredbig}.
\end{proof}

The two notions of equivalence do not coincide.
To show this, we make use of the following result.

\begin{theorem}\label{th:pseudo2}
Let $n=t t'$, for some $t,t' \in \mathbb{N}$ and let $s,s'\in \mathbb{N}$ be such that $\gcd(s,t)=\gcd(s',t)=1$. Let
$f(x)=\sum_{i=0}^{t'-1} a_i x^{q^{it+s}}$ and $g(x)=\sum_{i=0}^{t'-1} b_i x^{q^{it+s'}}$ two R-$q^t$-partially scattered polynomials.
Then $f(x)$ and $g(x)$ are weakly equivalent if and only if $s \equiv \pm s' \pmod{t}$.
\end{theorem}
\begin{proof}
If $f(x)$ and $g(x)$ are weakly equivalent then the $\fq$-linear sets $L_{U_f}$ and $L_{U_g}$ are $\mathrm{P\Gamma L}(t',q^t)$-equivalent.
By Theorem \ref{th:pseudo},  $L_{U_f}$ and $L_{U_g}$ are of pseudoregulus type and \cite[Theorem 3.7]{LMPT:14} implies that the companion automorphisms $\sigma_f$ and $\sigma_g$ of $f(x)$ and $g(x)$, respectively, are such that $\sigma_f=\sigma_g^{\pm 1}$, that is $s \equiv \pm s' \pmod{t}$.
Assume now that $s \equiv s' \pmod{t}$.
Consider the $\F_{q^t}$-linear map $F$ of $\F_{q^n}\times\F_{q^n}$ defined by
\[ F(x,y)=(x,g(f^{-1}(y))), \]
for every $x,y \in \fqn$.
Clearly, $F(U_f)=U_g$.
If $s \equiv -s' \pmod{t}$, then consider the $\F_{q^t}$-semilinear map $F$ of $\F_{q^n}\times\F_{q^n}$ defined by
\[ F(x,y)=(f^{-1}(y),g(x)), \]
for every $x,y \in \fqn$.
Again, $F(U_f)=U_g$.
\end{proof}

By Proposition \ref{prop:equivbin}, we know that there exist at least two non-equivalent binomials in the family of Proposition \ref{prop:family} when $t=3$ and $t'=2$.

\begin{corollary}\label{cor:different}
Let $t=3$ and $t'=2$ and let $f(x)$ and $g(x)$ be two R-$q^t$-partially scattered non-equivalent binomials belonging to  family of Proposition \ref{prop:family}. Then $f(x)$ and $g(x)$ are weakly equivalent.
\end{corollary}
\begin{proof}
Consider the $\fq$-linear sets $L_{U_f}$ and $L_{U_g}$ of $\mathrm{PG}(3,q^3)$. By Theorem \ref{th:pseudo},  $L_{U_f}$ and $L_{U_g}$ are of pseudoregulus type. Thus $f(x)$ and $g(x)$ are weakly equivalent by Theorem \ref{th:pseudo2}.
\end{proof}

Corollary \ref{cor:different} shows that the equivalence defined in Section \ref{sec:equiv} and the weak equivalence defined in this section are different.
Moreover, using Theorem \ref{th:pseudo2}, we can determine the number of weakly inequivalent R-$q^t$-partially scattered polynomials of the form \eqref{eq:form}.

\begin{corollary}
The number of weakly inequivalent R-$q^t$-partially scattered polynomials of the form \eqref{eq:form} is $\varphi(t)/2$, where $\varphi$ is the Euler totient function.
\end{corollary}

\section{Open problems}

We conclude the papers by pointing out some open problems.

\begin{itemize}
    \item In Proposition \ref{prop:LP} we characterize LP-polynomials which are L-$q^t$-partially scattered or R-$q^t$-partially scattered when $n$ is odd. The techniques developed in \cite{Zanella} may be useful to extend this characterization when $n$ is even.
    \item It would be interesting to find non-monomial exceptional L-$q^t$-partially scattered polynomials $f(x)$ which are not exceptional scattered.
    As already noted in Section \ref{sec:exceptionality}, this may be done through the investigation of the structure of the Galois group of $f(x)-sx^{q^\ell}$.
    \item Regarding the family of polynomials introduced in Section \ref{sec:example}, the weak equivalence issue has been completely solved in Section \ref{sec:weaker}, whereas less is known when considering the equivalence defined in Section \ref{sec:equiv}. Therefore it would be of interest to determine the equivalence classes of the family in Section \ref{sec:example} under the latter equivalence.
    \item Proposition \ref{prop:LZ} provides a practical machinery to construct families of R-$q^t$-partially scattered polynomials. Indeed, we use such a  proposition to construct the family in Section \ref{sec:example} by means of the R-$q^t$-partially scattered polynomials of monomial type. Other examples of R-$q^t$-partially scattered polynomials could arise from non-monomial ones.
\end{itemize}

\section*{Acknowledgments}

The authors are very grateful to Corrado Zanella for a careful reading of the paper and valuable suggestions.
This research  was supported by the Italian National Group for Algebraic and Geometric Structures and their Applications (GNSAGA - INdAM).
The second author is funded by the project ``Attrazione e Mobilità dei
Ricercatori'' Italian PON Programme (PON-AIM 2018 num. AIM1878214-2).
The second and the third authors are supported by the project ``VALERE: VAnviteLli pEr la RicErca" of the University of Campania ``Luigi Vanvitelli''.

\end{document}